\documentclass[12pt,a4paper,twoside, reqno]{amsart}
\usepackage{amsfonts, amsthm, amsmath, amssymb}
\usepackage{mathrsfs,amsmath}
\usepackage{hyperref}
\usepackage{blindtext}
\hypersetup{colorlinks=false}
\usepackage{comment}

\usepackage[margin=2.6cm]{geometry}
\usepackage{color}

\usepackage{helvet}

\newcommand{\A}{\mathcal{A}}
\newcommand{\B}{\mathcal{B}}

\usepackage{amsthm}
\newtheorem{theorem}{Theorem}
\newtheorem{lemma}{Lemma}

\newtheorem{definition}{Definition}

\newtheorem{proposition}{Proposition}
\newtheorem{corollary}{Corollary}
\newtheorem{example}{Example}

\begin{document}
	
	\author{Suraj Panigrahy}
	
	\title{ Short second moment of $\textrm{GL(2)}$ $L$-functions via Trivial Delta}
	\address{ Suraj Panigrahy \newline { Department of Mathematics and Statistics, Indian Institute of Technology, Bombay, India; \newline  
    Email: suraj.grahy12@gmail.com  
	} }		
	\date{\today}
\maketitle
    
    \begin{abstract}
		Let $f$ be a holomorphic cusp form for $SL(2,\mathbb{Z})$. In this paper, we prove for $T^{\frac{1}{3}} \ll M \ll T^{\frac{1}{2}}$, the following short interval second moment
        \begin{equation*}
           \int_T^{T+M}\left|L\left(\frac12+it,f\right)
           \right|^2 dt \ll_{f,\epsilon}
      T^\epsilon \left(    M + \frac{T}{M}+\sqrt{TM}\right).
        \end{equation*}
    The proof uses the trivial delta method along with conductor lowering technique.\end{abstract}
    \setcounter{tocdepth}{1}
\tableofcontents
\section{Introduction}\label{one}
The study of moments of $L$-functions is a central topic in analytic number theory. They measure the average size of an $L$-function on the critical line and have deep connections with subconvexity, non-vanishing, equidistribution of Fourier coefficients, and random matrix theory.
For the Riemann zeta-function, Hardy and Littlewood~\cite{Hali} proved in 1918 that
\begin{equation*}
    \int_{0}^{T}\left|\zeta\!\left(\frac12+it\right)\right|^2\,dt \sim T \log T,
\end{equation*}
establishing the first asymptotic formula for the second moment on the critical line. A decade later Ingham~\cite{Ing} obtained the asymptotic formula
\begin{equation*}
    \int_{0}^{T} \left|\zeta\left(\frac{1}{2}+it\right) \right|^4 dt \sim \frac{T}{2 \pi^2} \log^4 T,
\end{equation*}
and Heath-Brown~\cite{HB} later refined this by proving an asymptotic expansion with a power saving error term. Precisely he showed that
\begin{equation*}
    \int_{0}^{T} \left|\zeta\left(\frac{1}{2}+it\right)\right|^4 dt = \frac{T}{2 \pi^2} P_4(\log T) + O(T^{\frac{7}{8}+\epsilon}).
\end{equation*}
where $P_4$ is an explicit polynomial of degree $4$. Conrey and Ghosh~\cite{ConGho} conjectured that
\[
\int_0^T
\left|\zeta\!\left(\frac12+it\right)\right|^{2k}dt
\sim
c_k\,T(\log T)^{k^2},
\]
for suitable constants \(c_k\). Understanding the behaviour of $L$-functions on intervals whose length is smaller than $T$ is significantly more delicate. Over long intervals the oscillatory nature of the coefficients naturally produces substantial cancellation, whereas over short intervals much less cancellation is available. This philosophy was introduced by Iwaniec~\cite{IW1979} in his influential work on short moments of the Riemann zeta-function.

Rather than studying the fourth moment over the entire interval $[0,T]$, he proved 
\begin{align*}
    \int_{T}^{T+M} \left| \zeta\left(\frac{1}{2}+it\right)\right|^4 \,dt \ll MT^\epsilon \qquad \text{for }  M \geq T^{\frac{2}{3}+\epsilon}.
\end{align*}
His work revealed that short moment estimates naturally lead to subconvexity bounds through positivity. This idea has become one of the standard approaches to subconvexity: instead of estimating an individual value of an $L$-function directly, one estimates its average over a short interval and then deduces pointwise bounds.

Let $f$ be a holomorphic Hecke cusp form for $SL(2,\mathbb{Z})$ with normalized Hecke eigenvalues $\lambda_f(n)$. The associated $L$-function is defined by
\begin{equation*}
L(s,f)=\sum_{n=1}^{\infty}\frac{\lambda_f(n)}{n^s},
    \qquad \Re(s)>1.
\end{equation*}
In this paper we are interested in bounding the following short second moment
\begin{equation*}
S_f(M,T)=\int_T^{T+M}\left|L\left(\tfrac12+it,f\right)\right|^2 \,dt.
\end{equation*}
Such an estimate follows immediately from the asymptotic formula of Good~\cite{Good}. He investigates the second moment of $L(s,f)$ and proves the asymptotic formula
\[
\int_0^T
\left|
L\!\left(\frac{1}{2}+it,f\right)
\right|^2dt
=
2aT(\log T+b)
+
O((T\log T)^{2/3}),
\]
for suitable constants $a$ and $b$. Good's work may also be viewed as a cuspidal analogue of Iwaniec's treatment of short moments of the Riemann zeta-function. Good's argument is based on the spectral theory of automorphic forms. After applying the approximate functional equation, he studies shifted convolution sums by introducing an associated shifted Dirichlet series and analyzing its spectral decomposition via the theory of automorphic forms, which may be viewed as an application of the Kuznetsov trace formula. 

The main result of the paper is
\begin{theorem}\label{theorem1}
Let
\[
T^{1/3}\ll M \ll T^{1/2}.
\]
Then
\[
S_f(M,T)
\ll_{f,\epsilon}
T^\epsilon\left(M+\frac{T}{M}+\sqrt{TM}\right).
\]

In particular, taking $M=T^{1/3}$, we obtain
\[
\int_T^{T+T^{1/3}}
\left|
L\!\left(\frac12+it,f\right)
\right|^2dt
\ll_{f,\epsilon}
T^{2/3+\epsilon}.
\]
\end{theorem}

As an immediate consequence of Theorem~\ref{theorem1} and the positivity of the integrand, we recover the classical Weyl bound.

\begin{corollary}\label{theorem2}
We have
\[
L\!\left(\frac12+iT,f\right)
\ll_{f,\epsilon}
T^{1/3+\epsilon}.
\]
\end{corollary}
Our starting point is the approximate functional equation, which reduces the estimation of $S_f(M,T)$ to estimating  shifted convolution sums of the form
\begin{equation}\label{shn}
S(H,N)
=
\sum_{h\sim H}
\sum_{n\sim N}
\lambda_f(n)\lambda_f(n+h)
\left(1+\frac{h}{n}\right)^{iT}.
\end{equation}
The main novelty of this paper is an elementary treatment of the resulting shifted convolution sum. We separate the oscillations using the  trivial delta method. The resulting sums are then transformed through Poisson summation and the Voronoi summation formula, after which the integrals transforms are analysed by repeated integration by parts. Our approach is technically simpler while still yielding the Weyl exponent.
Using Lemma~\ref{ramanujan} together with Cauchy's inequality, one obtains the trivial estimate $
S(H,N)\ll N^{1+\varepsilon}H$. Combined with Equation~\eqref{combi}, this is already sufficient whenever $H\ll T^{1/6}$. For the remaining range, we establish the following bound. In fact, our argument proves the stronger statement that, for $
T^\epsilon \ll H \ll NT^\epsilon/M$. we have the following estimate.  
\begin{proposition}\label{boundshn}
Let $S(H,N)$ be the shifted convolution sum defined in equation \eqref{shn}.For $N\ll T^{1+\epsilon}$,
\[
S(H,N)
\ll
\begin{cases}
NHT^\epsilon
& 
\text{ for }N \ll T^{2/3+\epsilon},
\\[1.2ex]
NT^\epsilon \left(
\frac{\sqrt{T}}{\sqrt{M}}
+
\frac{T}{M^2}
\right)
&\text{ for }
T^{2/3+\epsilon} \ll N \ll T^{1+\epsilon}.
\end{cases}
\]
\end{proposition}
The remainder of the paper is organized as follows. In Section~\ref{two}, we outline the main ideas of the proof. Sections~\ref{three} and~\ref{four} collect the necessary background on holomorphic cusp forms and the analytic tools used throughout the paper. In Section~\ref{five}, we reduce the estimation of the short second moment to shifted convolution sums. The remaining sections establish the required estimate for the shifted convolution sums using the trivial delta method, Poisson summation, the Voronoi summation formula, and oscillatory integral estimates. We complete the proofs of Proposition~\ref{boundshn} and Theorem~\ref{theorem1} in Section~\ref{twelve}.
\subsection{Notations-} For any complex number $z$ we set $e(z) := e^{2\pi iz}$. $q$ and $q'$ are primes unless stated otherwise. By $q \in \mathfrak{Q}$, we mean $q$ is a prime number in the interval $[Q,2Q]$. By $A \asymp B$ we mean that $T^{-\epsilon}B \le A \le T^{\epsilon}B$. 
By $A \approx B$ we mean $c_1A \le B \le c_2A$ for some positive real $c_1, c_2$. By $A \sim B$ we mean that $B \le A < 2B$. 
By the notation $X \ll Y$ we mean that for any $\epsilon > 0$, there is a constant $c > 0$ such that $|X| \le cY$. The implied constants may depend on the cusp form $f$ and $\epsilon$. At various places the choice of $\epsilon > 0$ may be different.
\section{Outline of the Proof}\label{two}
For simplicity, we focus on the generic case $N=T$ and $H=T/M$. We invoke the trivial delta method in the following form:
\begin{equation*}
    \delta(n+h,m)=\frac{1}{|\mathfrak{Q}|} \sum_{q \in \mathfrak{Q}} \frac{1}{q} \sum_{a \,\text{mod}\, q} e\left(\frac{a(n+h-m)}{q}\right) \int_{K}^{2K} e\left(\frac{v(n+h-m)}{N}\right), 
\end{equation*}
to separate the oscillations $\lambda_f(m)$ and $\lambda_f(n) \left(1+\frac{h}{n}\right)^{iT}$ from the expression of $S(H,N)$. The contribution from the term $a=0$ is $
O\!\left(N^2H/Q\right)$, which is $O(T^{-A})$ for any $A>0$ upon choosing $Q$ sufficiently large. Hence it suffices to consider the contribution from $a \neq 0$, which is given by
\begin{align*}
      S(H,N)= \frac{1}{|\mathfrak{Q}|}\sum_{q \in \mathfrak{Q}}\frac{1}{q} \sideset{}{^*}\sum_{a \,\text{mod}\, q} \frac{1}{K} & \int \sum_h e\left( \frac{ah}{q}\right) e\left( \frac{hv}{N}\right)   \sum_n \lambda(n) e\left(\frac{an}{q}\right) e\left( \frac{nv}{N}\right)\left(1+\frac{h}{n}\right)^{iT}   \\
    &\times \sum_m\lambda(m) e\left( \frac{-ma}{q}\right) e\left( \frac{-mv}{N}\right) \mathrm{d}v.
   \end{align*}
\subsection{Dual Summation Formulas}
After the separating the variables with the delta method, we apply Poisson Summation to the $h$-sum, The initial length is $H.$ The Poisson transform has conductor $QK$, so
\begin{align*}
\text{dual length}
&=
\frac{\text{conductor}}{\text{initial length}}
=
\frac{KQ}{H},
&
\text{saving}
&=
\frac{\text{initial length}}{\sqrt{\text{conductor}}}
=
\frac{H}{\sqrt{KQ}}.
\end{align*}
Applying the $GL(2)$ Voronoi summation formula to both the $m$ and $n$ sums. In either summation, Each of them have initial length$=N$. Thus
\begin{align*}
    \text{dual length}&=\frac{\text{conductor}}{\text{initial length}} = \frac{Q^2K^2}{N}, &
\text{saving}&=\frac{\text{initial length}}{\sqrt{\text{conductor}}}=\frac{N}{QK}.
\end{align*}
\subsection{Analysis of Integral Transform}
After analysing the oscillatory integral arising from the Voronoi transform, the sum takes the form(upto normalizing factors)
\begin{align*}
     \sum_{q \in \mathfrak{Q}} \sum_{\substack{ h \sim \frac{KQ}{H} }}  \sum_{ m,n \asymp \frac{K^2q^2}{N}}   \lambda_f(n)  \lambda_f(m) \left(1+\frac{h}{n}\right)^{iT}  e\left(  \frac{-2 \sqrt{nT}}{\sqrt{2 \pi q h}}  \right) e\left(  \frac{2 \sqrt{mT}}{\sqrt{2 \pi q h}}  \right) e\left(\frac{\overline{h}n}{q}\right)e\left(-\frac{\overline{h}m}{q}\right).
\end{align*}
\subsection{$a \, \text{ mod } \, q$ and $v$-integral saving}
We save $\sqrt{Q}$ from the $a \, \text{ mod } \, q$ sum and $\sqrt{K}$ from the $v$-integral.
\subsection{Cauchy's Inequality}
Applying the Cauchy inequality separates the $m$- and $n$-variables, reducing the problem to estimating exponential sums of the form
\[
\sum_{\substack{q \asymp Q}} \sum_{\substack{h \asymp \frac{KQ}{H}}}
\left|
\sum_{n \asymp \frac{K^2Q^2}{N}}
e\left(\frac{\overline{h}n}{q}\right) e\left(-\frac{2\sqrt{nT}}{\sqrt{2\pi qh}}\right)
\right|^2.
\]
The phase $\frac{\overline{h}n}{q}$ is the arithmetic oscillation arising from the Voronoi summation formula, whereas the phase $-\frac{2\sqrt{nT}}{\sqrt{2\pi qh}}$, whose generic size is $K$, originates from the analysis of the integral transform.
\subsection{Duality principle}
We apply duality principle to reduce the problem from GL$(2)$ to GL$(1)$. We bring out the $n$ sum outside, the expression takes the form before we expand the absolute value squared for the expression
\begin{equation*}   \sup_{\| \alpha \|_2 =1} \sum_{\substack{ n \asymp \frac{K^2Q^2}{N} } } \Big|
    \sum_{\substack{q \\ q \asymp Q}} \sum_{\substack{h \\ h \asymp \frac{KQ}{H}}} \alpha(q,h)
    e\left( \frac{\overline{h}n}{q}- \frac{2 \sqrt{nT}}{\sqrt{2 \pi q h}}\right)\Big|^2.
\end{equation*}
\subsection{Poisson Summation}
In the zero frequency, we save the size of the diagonal which is precisely $\frac{KQ^2}{H}$ meanwhile in the nonzero frequency we save $\frac{K^2Q^2}{N\sqrt{K}}$ due to the presence of the additive characters in the $n$-sum.
Thus overall we save $$\min\left\{\frac{Q^2K}{H}, \frac{Q^2K^2}{N\sqrt{K}}\right\}.$$ because of the Poisson summation on $n$.
\subsection{Conclusion}
Combining the savings from the preceding steps gives an total saving of
\begin{equation*}
N^2H \min\left\{\frac{1}{HK}, \frac{1}{N \sqrt{K}}\right\}.
\end{equation*}
Thus we get $S(H,N) \ll HK+N\sqrt{K}$. Upon substituting $K=T/M$ with $H \ll N/M$ yields the following bound on the shifted convolution sum, given by
\begin{equation*}
      S(H,N) \ll_{f, \epsilon}  NT^\epsilon \left(
\frac{\sqrt{T}}{\sqrt{M}}
+
\frac{T}{M^2}
\right).
\end{equation*} This proves Proposition \ref{boundshn}.

\section{Preliminaries on Automorphic Forms}\label{three}
\subsection{Holomorphic Cusp Forms}
Let $k\geq2$ be an even integer and let $N\geq1$. Denote by
$S_k(N)$ the space of holomorphic cusp forms of weight $k$ and
level $N$. Every $f\in S_k(N)$ admits a Fourier expansion
\[
f(z)=\sum_{n=1}^{\infty}\psi_f(n)n^{\frac{k-1}{2}}e(nz),
\]
the absence of the constant term reflecting the cuspidality of $f$. The space $S_k(N)$ is equipped with the Petersson inner product and
admits an orthogonal basis consisting of simultaneous eigenfunctions of
the Hecke operators. Let $f\in S_k(N)$ be a normalized Hecke eigenform satisfying
$\psi_f(1)=1$.
Its Hecke eigenvalues are given by $\lambda_f(n)=\psi_f(n)$
and satisfy $\lambda_f(mn)=\lambda_f(m)\lambda_f(n)$ for $ (m,n)=1.$
The $L$-function associated to $f$ is defined by
\[
L(s,f)=\sum_{n=1}^{\infty}\frac{\lambda_f(n)}{n^s},
\qquad \text{ for }
\Re(s)>1.
\]
By Deligne's bound, this Dirichlet series converges absolutely in the
half-plane $\Re(s)>1$ and admits the Euler product $
L(s,f)=
\prod_p
(1-\lambda_f(p)p^{-s}
+p^{-2s})^{-1}.
$
The completed $L$-function satisfies a functional equation relating the values at $s$ and $1-s$. In the $t$-aspect, its analytic conductor is
of size $(1+|t|)^2$.
\subsection{Approximate Functional Equation} The following approximate functional equation is a standard consequence of the functional equation.
\begin{lemma}[Approximate Functional Equation, \cite{IK}, Theorem-5.3]\label{afe}
Let $L(s,f)$ be an $L$-function. Let $G(u)$ be any function which is holomorphic and bounded in the strip $-4 < \Re(u) < 4$, even, and normalized by $G(0)=1$. Let $X>0$. Then for $s$ in the strip $0 \le \sigma \le 1$ we have
\begin{equation}
L(s,f)=\sum_{n=1}^{\infty} \frac{\lambda_f(n)}{n^s} V_s\!\left(\frac{n}{X\sqrt {q(f)}}\right) +\epsilon(f,s)\sum_{n=1}^{\infty}\frac{\overline{\lambda_f(n)}}{n^{1-s}}V_{1-s}\!\left(\frac{nX}{\sqrt{ q(f)}}\right),
\end{equation}
where $q(f)$ is the conductor of $L(s,f)$ and $V_s(y)$ is a smooth function defined by
\begin{equation}
V_s(y)=\frac{1}{2\pi i}\int_{(3)} y^{-u}G(u)\frac{\gamma(f,s+u)}{\gamma(f,s)}\frac{du}{u},
\end{equation}
and
\begin{equation}
\epsilon(f,s)=\epsilon(f)\,q(f)^{\frac12-s}\frac{\gamma(f,1-s)}{\gamma(f,s)}.
\end{equation}
\end{lemma}
\subsection{Inert Functions}
We frequently encounter families of smooth functions satisfying uniform derivative bounds. Such families are called \emph{inert functions}. Throughout this subsection, we follow the exposition of Kiral, Petrow and Young~\cite{KIY}, recalling the definition together with several basic examples that will be used later.
\begin{definition}
Let $\mathcal{F}$ be an index set, and let $
F:\mathcal{F}\rightarrow\mathbb{R}_{\ge1}
$
be a function. A family of smooth functions $\{W_T\}_{T \in \mathcal{F}}$ supported on a product of dyadic intervals in $\mathbb{R}^d_{ >0}$ is called $F_T$-inert if, for each $j=(j_1, j_2 \dots j_d) \in \mathbb{Z}^d_{ \geq 0}$, we have
\begin{equation*}
   \sup_{T \in \mathcal{F} } \text{ }  \text{ }\sup_{(x_1, x_2 \dots x_d) \in \mathbb{R}^d_{>0}} F_T^{-|j|} |\mathbf{x}^j  W^{(j)}_T(\mathbf{x})| < \infty,
\end{equation*}
where
\[
|j|=j_1+\cdots+j_d,\qquad
\mathbf{x}=(x_1,\ldots,x_d),\qquad
\mathbf{x}^j=x_1^{j_1}\cdots x_d^{j_d}.
\]    
\end{definition}

We frequently encounter several inert functions throughout the paper, thus we begin with few important examples.
\begin{example}
    Let $T \in \mathbb{R}_{>0}=\mathcal{F}$, Let $W(t)$ be a fixed smooth function supported on $[1,2]$, Then the family $\{W_T\}_{T \in \mathcal{F}}$ where $W_T$ is defined by
    $W_T(t)=W\left(\frac{t}{T}\right)$ is $1$-inert.
\end{example}
\begin{example}\label{dil}
    Let $(T_1,T_2) \in  \mathbb{R}^2_{>0}=\mathcal{F}$, Let $W(t_1,t_2)$ be a fixed smooth function supported on $[1,2] \times [1,2]$, Then the family $\{W_{T_1, T_2}\}_{(T_1,T_2) \in \mathcal{F}}$ where $W_{T_1, T_2}$ is defined by
    $$W_{T_1,T_2}(t_1,t_2)=W\left(\frac{t_1}{T_1},\frac{t_2}{T_2}\right),$$ is $1$-inert.
\end{example}
\begin{example}\label{expo}
    Let $T_\lambda=(T_1,T_2,\dots,T_d,\lambda_1,\lambda_2,\dots,\lambda_d)\in (\mathbb{R}_{>0})^d\times\mathbb{R}^d=\mathcal{F}$. Then the family $\{E_{T_\lambda}\}_{T_\lambda\in\mathcal{F}}$, where $E_{T_\lambda}$ is defined by
    \[
    E_{T_\lambda}(t_1,t_2,\dots,t_d)
    =
    e^{i\lambda_1t_1+i\lambda_2t_2+\cdots+i\lambda_dt_d},
    \]
    is $\max\{|\lambda_1|T_1,|\lambda_2|T_2,\dots,|\lambda_d|T_d\}$-inert.
\end{example}
\begin{example}\label{ex3}
    Let $\{W_T\}_{T \in \mathcal{F}}$ be $F_T$- inert family  and $\{V_{T'}\}_{T' \in \mathcal{F}'}$ be $G_{T'}$- inert, Then the family $$\{W_T \cdot V_{T'}\}_{(T,T') \in \mathcal{F} \times \mathcal{F}'},$$ is $\max\{F_T, G_{T'}\}$ inert.
\end{example}
\begin{example}\label{weightfunction}
    Let $T_\lambda=(T_1,T_2, \dots T_d, \lambda_1, \lambda_2 \dots \lambda_d) \in  \mathbb{R}^d_{>0} \times \mathbb{R}^d=\mathcal{F}$, Let $W(t_1,t_2 \dots t_d)$ be a fixed smooth function supported on $[1,2]^d$, Then the family $\{W_{T_\lambda}\}_{T_\lambda \in \mathcal{F}}$ where $W_{T_\lambda}$ is defined by
    $$W_{T_\lambda}(t_1,t_2 \dots t_d)=e^{i\lambda_1 t_1+i\lambda_2 t_2 \dots i\lambda_d t_d} \text{ } W\left(\frac{t_1}{T_1},\frac{t_2}{T_2} \dots \frac{t_d}{T_d} \right),$$ is $\max\{1, |\lambda_1| T_1, |\lambda_2| T_2, \dots |\lambda_d| T_d\}$ inert. This follows immediately from Examples ~\ref{dil},  ~\ref{expo} and~\ref{ex3}.
    
\end{example}
\begin{example}\label{fourier}
Let $\{W_T\}_{T\in\mathcal{F}}$ be an $F_T$-inert family which is supported on $[1,2]$. Define
\[
\widehat{W}_T(\xi)=\int_{\mathbb R}W_T(t)e(-t\xi)\,dt.
\]
 Then an application of intergation by parts shows that the family $\{\widehat{W}_T\}_{T \in \mathcal{F}}$ satisfies the derivative bounds required for an $F_T$-inert family. However it need not be compactly supported. This can be achieved by an application of smooth partition of unity to $\xi$.
\end{example}

\subsection{Voronoi Summation Formula}
\begin{lemma}[GL(2) Voronoi Summation Formula]\label{voronoi}
Let $f$ be a holomorphic cusp form of level $1$ with Fourier
coefficients $\lambda_f(n)$. Let $a$ and $q$
be integers with $(a,q)=1$. Let $G$ be a compactly supported smooth
function on $\mathbb{R}$. Then
\begin{equation}
\sum_{m=1}^{\infty}\lambda_f(m)e\!\left(-\frac{am}{q}\right)G(m)=\frac{2 \pi i^k}{q}\sum_{m=1}^{\infty}\lambda_f(m)e\!\left(\frac{\overline{a}m}{q}\right)
I\!\left(m\right),\end{equation} where
\[I(y)=\int_0^\infty
G(x)J_{k-1}\left(\frac{4\pi\sqrt{xy}}{q}\right)
\,dx,
\]
where $J_{k-1}$ denotes the $J$-Bessel function.
\end{lemma}
\subsection{Ramanujan's Bound on Average}
\begin{lemma}\label{ramanujan}
Let $\lambda_f(n)$ be the Fourier coefficents of the cusp form $f$. Then for any $\epsilon>0$,
\begin{equation*}
\sum_{n \ll X}\left|\lambda_f(n)\right|^2
\ll_{f,\epsilon}
X^{1+\epsilon}.
\end{equation*}
\end{lemma}
\section{Tools Used}\label{four}
\subsection{Delta Method}

The following variant of the $\delta$-method was first introduced byAggarwal~ \cite{Agg}. Let $\delta:\mathbb{Z}\rightarrow\{0,1\}$ be defined by
\[
\delta(n)=
\begin{cases}
1,& n=0,\\
0,& \text{otherwise}.
\end{cases}
\]
Assume $Q \gg NT^\epsilon/K$ and let $q \sim Q$. Then $0<|n|\ll NT^\epsilon/K$, there is no integer $q \sim Q$ dividing $n$. For $|n|\gg NT^\epsilon/K$, we insert a smooth integral transform. The rapid decay of the Fourier transform of the smooth weight then shows that the contribution from this range is $O_A(T^{-A})$ for every $A>0$. This gives the following lemma.

\begin{lemma}\label{delta}
Let $B\in C_c^\infty(\mathbb R)$ satisfy $\int_{\mathbb R}B(v)\,dv=1$, and let $\mathfrak Q$ denote the set of primes in $[Q,2Q]$, where $Q>NT^\epsilon/K$. Then
\[
\delta(n)=\frac{1}{|\mathfrak{Q}|}\sum_{q\in\mathfrak Q}\frac1q\sum_{a\bmod q}e\!\left(\frac{an}{q}\right)\frac1K\int_{\mathbb R}B\!\left(\frac vK\right)e\!\left(\frac{nv}{N}\right)\,dv+O_A(T^{-A}).
\]
\end{lemma}
\subsection{Smooth Partition of Unity}
We often use dyadic partition of unity, 
to localise an infinite sum smoothly inside an interval $n \in [N,2N]$, We use the following construction of a smooth dyadic partition of unity, following Kumar, Mallesham, Sharma and Singh~\cite{SMPS}. For completeness, we include the details of the construction.
\begin{lemma}
    There exists a smooth function $W: \mathbb{R} \to \mathbb{R}_{\geq 0}$ such that
    \begin{enumerate}
        \item Support of $W$ lies in $[\frac{3}{4},2].$ 
        \item $W(x)=1$ for all $x \in [1,3/2].$
        \item $W(x)+W(x/2)=1$ for all $x \in [1,3].$
    \end{enumerate}
\end{lemma}
\begin{proof}
See Li ~\cite{Xi}
\end{proof}
For any $H \in \mathbb{N}$, define
$$F(x)=W(x)+W\left(\frac{x}{2}\right)+W\left(\frac{x}{4}\right) \dots +W\left(\frac{x}{2^H}\right).$$
If $W(x)W(\frac{x}{4}) \neq 0$ implies $x \in [3/4,2]$ and $x \in [3,8]$ which is a contradiction. Thus $W(x)W(\frac{x}{4})=0.$ for all $x \in \mathbb{R}.$ If $x \in [1,3 \cdot 2^{H-1}]$, then there exists $0 \leq k <H$ such that $2^k \leq x \leq 3 \cdot 2^k$. Consequently
\begin{equation*}
    1=W\left(\frac{x}{2^k}\right)+W\left(\frac{x}{2^{k+1}}\right).
\end{equation*}
This implies $F(x)=1$ on $[1,3 \cdot 2^{H-1}].$ Morever it follows directly from the definition that $F(x)$ is supported on $[3/4, 2^{H+1}]$.  We conclude the existence of a smooth partition of unity since \begin{equation*}
    \sum_{k \in \mathbb{Z}} W\left(\frac{x}{2^k}\right)=1 \qquad \text{ for all }  x \in [1,\infty).
\end{equation*}
\subsection{Duality Principle}
Duality helps us to separate the Fourier coefficents which can be bounded by Ramanujan Sums on Average.(See Lemma $\ref{ramanujan}$)
\begin{lemma}[Duality principle]\label{dualitylemma}
		Let $\phi: \mathbb{Z}^2\rightarrow \mathbb{C}$. For any complex numbers $a_m$, \begin{align*}
			\sum_n\left|\sum_m a_m\phi(m,n)\right|^2\ll \left( \sum_{ m} |a_m|^2\right)\sup_{\|\beta\|_2=1}\sum_m\left|\sum_n\beta(n)\phi(m,n)\right|^2,
		\end{align*}
		where the supremum is taken over all sequences of complex numbers $\beta(n)$ such that $$\|\beta\|_2=\sqrt{\sum_n|\beta(n)|^2}=1.$$
	\end{lemma}

\subsection{Integration by Parts}
We frequently encounter oscillatory integrals of the form $\int_{-\infty}^{\infty} w(t)e^{ih(t)}\,dt.
$ We begin by recalling a standard integration by parts lemma, which provides sufficient conditions under which such integrals are negligibly small. For the proof, refer Lemma $8.1$ in Blomer, Khan and Young~\cite{BKY}.
\begin{lemma}\label{ibparts}
Let $Y \geq 1$  and $X, Q, U, R>0$, and let 
\begin{equation*}
    I=\int_{-\infty}^{\infty} w(t)e^{ih(t)}\,dt,
\end{equation*}
where $w$ is a smooth function supported on an interval $[\alpha,\beta]$ satisfying $w^{(j)}(t)\ll_j X U^{-j},$
and $h$ is a smooth real-valued function on $[\alpha,\beta]$ satisfying $|h'(t)|\ge R,$
and $
h^{(j)}(t)\ll_j YQ^{-j}$ for all $j\ge2.$ Then, for every integer $ A\ge0$,
\begin{equation*}
I\ll_A (\beta-\alpha) X R^{-A}\left(U^{-A}+\left(\frac{Q}{\sqrt{Y}}\right)^{-A} \right).    
\end{equation*}
\end{lemma}
\subsection{Poisson Summation Formula}\label{psf}

Let $f:\mathbb{R}\rightarrow\mathbb{C}$ be a Schwartz class function, Define the Fourier transform of $f$ by
\[ \widehat{f}(y)=\int_{\mathbb{R}}f(x)e(-xy)\,dx.
\]
\begin{lemma}
For any function $f$ in the Schwartz class
\[
\sum_{n\in\mathbb{Z}}f(n)=\sum_{n\in\mathbb{Z}}\widehat{f}(n).
\]
\end{lemma}
As a consequence of Poisson Summation Formula, we have
\begin{lemma}\label{poissonmodq}
$f:\mathbb{R}\rightarrow\mathbb{C}$ be a Schwartz class function
\begin{align*}
    \sum_{h\in\mathbb Z}F(h)e\!\left(\frac{ha}{q} \right)=\sum_{\substack{h\in\mathbb Z \\ h \equiv -a \,\text{mod}\, q}}   \int_{\mathbb{R}} e\left( \frac{-hy}{q}\right) F(y) \,dy.
\end{align*}
\end{lemma}
\subsection{Second Deriviative Bound}
For our purpose, The second deriviative bound suffices. One can refer Huxley\cite{HUX} for the proof. We state the lemma as follows.
\begin{lemma}{(Second Deriviative Test)}\label{secondd}
    Let $f(x)$ be a real valued and twice differentiable on the open interval $(a,b)$ with $f''(x) >0$ on $(\alpha, \beta)$. Let $g(x)$ be real, and let $V$ be the total variation of $g(x)$ on the closed interval $[\alpha, \beta]$ plus the maximum modulus of $g(x)$ on $[\alpha, \beta]$. Then
    \begin{align*}
        \Big| \int_{\alpha}^\beta g(x) e(f(x)) \,dx \Big| \leq \frac{V}{\sqrt{\pi \lambda}}.
    \end{align*}
\end{lemma}
\section{Setup}\label{five}

$L(\frac{1}{2}+it,f)$ has analytic conductor $T^2$ hence from the approximate functional equation(See Lemma ~\ref{afe}), it suffices to bound
\begin{equation*}
    S_f(M,T) \ll \int_{\mathbb{R}} U\left(\frac{t-T}{M}\right) \Big| \sum_{n=1}^{\infty} \frac{\lambda_f(n)}{n^{\frac{1}{2}+it}} V_{\frac{1}{2}+it}\left(\frac{n}{T}\right)  \Big|^2 dt,
\end{equation*}
where $V_s(x)$ is a smooth function satisfying $x^j V^{(j)}_s(x) \ll_{A,j} (1+|x|)^{-A}$ for any $j \geq 0$ and $A >0$. Integration by parts to the integral representation of $V_s\left(\frac{n}{T}\right)$ gives arbitrary savings for $n \geq T^{1+\epsilon}$. It follows that
\begin{equation*}
     S_f(M,T)=\int_{\mathbb{R}} U\left(\frac{t-T}{M}\right) \Big| \sum_{n \leq T^{1+\epsilon}} \frac{\lambda_f(n)}{n^{\frac{1}{2}+it}} V_{\frac{1}{2}+it}\left(\frac{n}{T}\right)  \Big|^2 dt    + O(T^{-A}).
\end{equation*}
We use a smooth dyadic partition of unity, We arrive at
\begin{align*}
     S_f(M,T) &\ll \int_{\mathbb{R}} U\left(\frac{t-T}{M}\right) \Big| \sum_{N \ll T^{1+\epsilon}} \sum_{n=1}^\infty \frac{\lambda_f(n)}{n^{\frac{1}{2}+it}}  W\left(\frac{n}{N}\right) V_{\frac{1}{2}+it}\left(\frac{n}{T}\right)  \Big|^2 dt     + O(T^{-A}) \\
     &\ll \sup_{N \ll T^{1+\epsilon}}  \frac{1}{N}\int_{\mathbb{R}} U\left(\frac{t-T}{M}\right) \Big| \sum_{n=1}^\infty \frac{\lambda_f(n)}{n^{it}} \frac{N^\frac{1}{2}}{n^\frac{1}{2}}  W\left(\frac{n}{N}\right) V_{\frac{1}{2}+it}\left(\frac{n}{T}\right)  \Big|^2 dt     + O(T^{-A}) \\
     & \ll \sup_{N \ll T^{1+\epsilon}} \frac{1}{N}\int_{\mathbb{R}} U\left(\frac{t-T}{M}\right) \Big|  \sum_{n=1}^\infty \frac{\lambda_f(n)}{n^{it}} W_N\left(\frac{n}{N}\right)  \Big|^2 dt    + O(T^{-A}),
\end{align*}
where we define $W_N(t)=t^{-\frac{1}{2}} W(t) V_s\left(\frac{Nt}{T}\right)$ and $W_N$ is supported in $[1,2]$ and satisfies $t^j W^{(j)}_N(t) \ll_j 1$ for $j \geq 0$. Now we expand the absolute value squared, and make a change of variable $t \to Mt+T$ to get
\begin{align*}
    S_f(M,T) &\ll \sup_{N \ll T^{1+\epsilon}} \frac{1}{N}\int_{\mathbb{R}} U\left(\frac{t-T}{M}\right) \sum_{m=1}^\infty \sum_{n=1}^\infty \frac{\lambda_f(n)}{n^{it}}\frac{\overline{\lambda_f(m)}}{\overline{m^{it}}} V\left(\frac{n}{N}\right) \overline{V\left(\frac{m}{N}\right)} dt     + O(T^{-A}) \\
    &\ll \sup_{N \ll T^{1+\epsilon}} \frac{M}{N} \sum_{m=1}^\infty \sum_{n=1}^\infty \lambda_f(n)\overline{\lambda_f(m)} V\left(\frac{n}{N}\right) \overline{V\left(\frac{m}{N}\right)} \left(\frac{m}{n} \right)^{iT} \int_{\mathbb{R}}   U\left(t\right) \left(\frac{m}{n} \right)^{iMt}  dt     + O(T^{-A}).
\end{align*}
Integration by parts to the $t$ integral shows gives arbitrary savings unless $|m-n| \ll \frac{NT^{\epsilon}}{M}.$ We make a change of variables $h=m-n$, this yields
\begin{align*}
    S_f(M,T) &\ll \sup_{N \ll T^{1+\epsilon}} \frac{M}{N} \Bigg| \, \sum_{|h| \ll \frac{NT^{\epsilon}}{M} } \sum_{n=1}^\infty \lambda_f(n)\overline{\lambda_f(n+h)} V\left(\frac{n}{N}\right) \overline{V\left(\frac{n+h}{N}\right)} \left(1+\frac{h}{n} \right)^{iT} \\
    & \times \int_{\mathbb{R}}   U\left(t\right) \left(1+\frac{h}{n} \right)^{iMt}  dt  \,\Bigg|   + O(T^{-A}). 
\end{align*}
For $h=0$, Ramanujan bound on average gives
\begin{equation}\label{h0contri}
    \sum_{n=1}^\infty |\lambda_f(n)|^2 \Big|V\left(\frac{n}{N}\right)\Big|^2  \ll NT^{\epsilon}.
\end{equation}
Thus
\begin{align*}
    S_f(M,T) &\ll MT^\epsilon +\sup_{N \ll T^{1+\epsilon}} \frac{M}{N}  \Bigg| \,\sum_{0<|h| \ll \frac{NT^{\epsilon}}{M} } \sum_{n=1}^\infty \lambda_f(n)\overline{\lambda_f(n+h)} V\left(\frac{n}{N}\right) \overline{V\left(\frac{n+h}{N}\right)} \left(1+\frac{h}{n} \right)^{iT} \\
    & \times \int_{\mathbb{R}}   U\left(t\right) \left(1+\frac{h}{n} \right)^{iMt}  dt \,\Bigg|    + O(T^{-A}).
\end{align*}
Let $\phi_{\pm}(h)=\phi(\pm h )$ where $\phi$ is a fixed function in $ C^\infty_c\left[\frac{1}{2},2\right].$ We perform a smooth dyadic subdivision for the $h$-sum with $h \neq 0$, This gives
\begin{align}\label{combi}
       S_f(M,T) &\ll MT^\epsilon +\sup_{N \ll T^{1+\epsilon}} \frac{M}{N}\sup_{ H \ll \frac{NT^{\epsilon}}{M} }  \left|\sum_{\pm }S(H,N) \right|    + O(T^{-A}),
\end{align}
where we define the shifted convolution sum 
\begin{equation*}
S(H,N)=    \sum_{h \in \mathbb{Z}}\sum_{n=1}^\infty \lambda_f(n) \overline{\lambda_f(n+h)} \left(1+\frac{h}{n}\right)^{iT} W_\pm\left(\frac{h}{H},\frac{n}{N}\right),
\end{equation*}
where $W_\pm(x,y)$ is a $T^{\epsilon}$-inert function supported on $C_c^\infty([1,2] \times \pm[1,2])$ defined by
\begin{equation*}
W_\pm(x,y)=\phi_{\pm} \left(x\right)V\left(y\right) \overline{V\left(y+\frac{Hx}{N}\right)} \int_{\mathbb{R}}   U\left(t\right) \left(1+\frac{Hx}{Ny} \right)^{iMt}  dt.
\end{equation*}
It would be sufficent to exploit the oscillations in $S(H,N)$ and get a sharper bound. We do this in the next sections. For $N \ll T^{\frac{2}{3}+\epsilon},$ the trivial bound suffices, Thus we assume $T^{\frac{2}{3}+\epsilon} \ll N.$

\section{Invoking Delta and Conductor Lowering}\label{six}
Let $
U\in C_c^\infty\!\left(\left[\frac12,\frac72\right]\right)$
be a fixed smooth function satisfying
$U(t)=1$ for all $ t \in [1,3].$
Then \begin{align*}
S(H,N)
&=\sum_{h \in \mathbb{Z}}\sum_{n=1}^{\infty}\sum_{m=1}^{\infty}
\lambda_f(n)\lambda_f(m)
\left(1+\frac{h}{n}\right)^{iT}
U\!\left(\frac{m}{N}\right)
W_\pm\!\left(\frac{h}{H},\frac{n}{N}\right)
\delta(m,n+h).
\end{align*}
We now detect the constraint $m=n+h$ using the delta method. We use delta method(See Lemma $\ref{delta}$) with conductor $Q>NT^{\epsilon}/K$,
\begin{align*}
S(H,N)
&=\frac{1}{|\mathfrak{Q}|K}
\sum_{q\in\mathfrak{Q}}\frac{1}{q}
\sum_{a=0}^{q-1}
\int_{\mathbb{R}}
B\!\left(\frac{v}{K}\right)
\sum_{h \in \mathbb{Z}}\sum_{n=1}^{\infty}\sum_{m=1}^{\infty}
\lambda_f(n)\lambda_f(m)
\left(1+\frac{h}{n}\right)^{iT} \\
&\qquad\times
U\!\left(\frac{m}{N}\right)
W_\pm\!\left(\frac{h}{H},\frac{n}{N}\right)
e\!\left(\frac{a(n+h-m)}{q}\right)
e\!\left(\frac{v(n+h-m)}{N}\right)
\,dv.
\end{align*}
We separate the contributions from the residue classes $a=0$ and $a\neq0$, obtaining
\begin{equation*}
    S(H,N)=S_1(H,N)+O\left(\frac{N^2H}{Q}+T^{-A}\right), 
\end{equation*}
where
\begin{align*}
S_1(H,N)&=\sum_{ h \in \mathbb{Z}} \sum_{ n=1}^\infty \sum_{ m=1}^{\infty} \lambda_f(n) \lambda_f(m) \Big(1+\frac{h}{n}\Big)^{iT} U\left(\frac{m}{N}\right)W_\pm\left(\frac{h}{H}, \frac{n}{N}\right) \\ & \times \frac{1}{|\mathfrak{Q}|K} \sum_{q \in \mathfrak{Q}} \frac{1}{q} \sum_{a=1}^{q-1} e\left(\frac{a}{q}(n+h-m)\right) \int_\mathbb{R} B\left( \frac{v}{K}\right) e\left( \frac{v}{N}(n+h-m)\right) \,dv .
\end{align*}
\section{Dual Summation Formulas}\label{seven}
The analysis for $W_+$ and $W_-$ is identical up to minor modifications. Hence it suffices to consider the contribution from $W_+$. We  apply Poisson summation to the $h$-sum and Voronoi summation formula to $m$ and $n$-sum
\subsection{Analysis of $h$ sum} We begin by applying the Poisson summation formula to the $h$-sum modulo $q$. This gives
\begin{align*}
    \sum_{h \in \mathbb{Z}} e\left( \frac{hv}{N}\right) \left(1+\frac{h}{n}\right)^{iT} W_{\pm}\left(\frac{h}{H}, \frac{n}{N}\right) e\left(\frac{ha}{q}\right) = H\sum_{\substack{h \in \mathbb{Z} \\ h +a\equiv 0 \,\text{mod}\, q }} \int_{\mathbb{R}} & e\left(\frac{vy_1H}{N}\right) e\left( -\frac{  y_1 h H}{q}\right)  \\ & \times\left(1+\frac{Hy_1}{n}\right)^{iT}  W_{+}\left(y_1, \frac{n}{N}\right) \,dy_1.
\end{align*}
Since $h \equiv -a \,\text{mod}\, q$, and $(a,q)=1$, implies $(h,q)=1$. An integration by parts arguement over the $y_1$ integral yields the contribution is negligible unless $h \asymp \frac{TQ}{N}=H_0$(say).
\subsection{Analysis of $m$ sum}
Define 
\begin{equation*}
    V_1(a,q)=\sum_{m=1}^\infty U\left(\frac{m}{N}\right) \lambda_f(m)e\left(- \frac{mv}{N}\right)e\left(- \frac{ma}{q}\right).
\end{equation*}
We apply Voronoi Summation Formula over the $m$-sum
\begin{equation*}
   V_1(a,q)= \frac{2 \pi  i^k}{q} \sum_{m=1}^\infty \lambda_f(m) e\left(\frac{ m \overline{a}}{q}\right) I_1(m,q) ,
\end{equation*}
where \begin{equation*}
    I_1(m,q)= \int_{\mathbb{R}} U\left(\frac{y_3}{N}\right) e\left(- \frac{vy_3}{N}\right) J_{k-1} \left( \frac{4 \pi \sqrt{my_3}}{q}\right) \,dy_3.
\end{equation*}
The weight of the form $f$ is $k$, and $J_{k-1}$ is the $(k-1)$ Bessel function. Using the standard properties we express 
\begin{equation*}
    J_{k-1}(4 \pi z)=e(2z) W_{k-1}(2z)+e(-2z) \overline{W}_{k-1}(2z).
\end{equation*}
Thus we have
\begin{align*}
    \sum_{m=1}^\infty U\left(\frac{m}{N}\right) \lambda_f(m)e\left(-\frac{mv}{N}\right)e\left(-\frac{ma}{q}\right)&= \frac{2 \pi N  i^k}{q} \sum_{\pm} \sum_{m=1}^\infty \lambda_f(m) e\left(\frac{m\overline{a}}{q}\right)   \\ & \times \int_{\mathbb{R}} U\left(y_3\right) e\left(- v y_3 \right) e\left(\pm\frac{2 \sqrt{mNy_3}}{q}\right) W^{\pm}_{k-1} \left( \frac{2 \sqrt{mNy_3}}{q}\right) \,dy_3,
\end{align*}
where $W_{k-1}^+(z)=W_{k-1}(z)$ and $W_{k-1}^-(z)=\overline{W}_{k-1}(z)$ where $W_{k-1}(z)$ is a smooth function satisying $z^j W_{k-1}^{(j)}(z) \ll z^{-1/2}$ for all $j \geq 0$. we have made a change of variable inside the integral to get the last expression. Define $W^\pm_1(z)= \sqrt{z} W_{k-1}^\pm(z)$, Thus we get $V_1(a,q)$ to be
 \begin{align*}
= &\frac{2 \pi N  i^k}{N^{1/4} \sqrt{q}} \sum_{\pm} \sum_{m=1}^\infty \frac{\lambda_f(m)}{m^{1/4}} e\left( \frac{m\overline{a}}{q}\right)   \int_{\mathbb{R}} U\left(y_3\right) e\left(- v y_3 \right) e\left(\pm\frac{2 \sqrt{mNy_3}}{q}\right) W_1^{\pm}\left( \frac{2 \sqrt{mNy_3}}{q}\right) y^{-1/4}_3 \,dy_3 \\
&=\frac{2 \pi N  i^k}{N^{1/4} \sqrt{q}} \sum_{\pm} \sum_{m=1}^\infty \frac{\lambda_f(m)}{m^{1/4}} e\left( \frac{m\overline{a}}{q}\right)   \int_{\mathbb{R}}  e\left(- v y_3 \right) e\left(\pm\frac{2 \sqrt{mNy_3}}{q}\right) U'_{m,q}(y_3) \,dy_3,
\end{align*}
where we define the new smooth weight function
\begin{equation*}
   U^\pm_{m,q}(y_3)= U\left(y_3\right) W_1^{\pm}\left( \frac{2 \sqrt{mNy_3}}{q}\right) y^{-1/4}_3.
\end{equation*}
Integration by parts to the $y_3$ integral shows that the contribution to $V_1(a,q)$ by the $m$-sum is negligible unless $v \asymp \frac{\sqrt{mN}}{q}$, Thus we must have $m \asymp \frac{K^2Q^2}{N}.$
\subsection{Analysis of $n$ sum}
Define 
\begin{equation*}
    V_2(a,q)=\sum_{n=1}^\infty \lambda_{f}(n) \left(1+\frac{Hy_1}{N}\right)^{iT} e\left(\frac{nv}{N}\right) W_{+}\left(y, \frac{n}{N}\right)e\left( \frac{na}{q}\right).
\end{equation*}
Application of Voronoi summation formula on the $n$-sum gives
\begin{align*}
    V_2(a,q)=\frac{2 \pi i^k}{q} \sum_{n=1}^\infty \lambda_{f}(n) e\left(- \frac{n\overline{a}}{q}\right) I_2(n),
\end{align*}
where \begin{equation*}
    I_2(n)=\int_{\mathbb{R}} e\left(\frac{vy_2}{N}\right) \left(1+\frac{Hy_1}{y_2}\right)^{iT} W_{+}\left(y, \frac{y_2}{N}\right) J_{k-1}\left(\frac{4 \pi \sqrt{ny_2}}{q}\right) \,dy_2.
\end{equation*}
The weight of the form $f$ is $k$, and $J_{k-1}$ is the $(k-1)$ Bessel function. Using the standard properties we express 
\begin{equation*}
    J_{k-1}(4 \pi z)=e(2z) W_{k-1}(2z)+e(-2z) \overline{W}_{k-1}(2z).
\end{equation*}
Thus we have
\begin{align*}
    V_2(a,q)= \frac{2 \pi N  i^k}{q} \sum_{\pm}  \sum_{n=1}^\infty \lambda_{f}(n) e\left(-n \frac{\overline{a}}{q}\right) & \int_{\mathbb{R}} e\left(vy_2\right) \left(1+\frac{Hy_1}{Ny_2}\right)^{iT} W_{+}(y_1, y_2) \\ & \times e\left(\pm \frac{2 \sqrt{nNy_2}}{q}\right) W^{\pm}_{k-1}\left(\frac{2 \sqrt{nNy_2}}{q}\right) \,dy_2,
\end{align*}
where $W_{k-1}^+(z)=W_{k-1}(z)$ and $W_{k-1}^-(z)=\overline{W}_{k-1}(z)$ where $W_{k-1}(z)$ is a smooth function satisying $z^j W_{k-1}^{(j)}(z) \ll \frac{1}{\sqrt{z}}$ for all $j \geq 0$. We have made a change of variable inside the integral to get the last expression. Define $W^\pm_1(z)= \sqrt{z} W^\pm_{k-1}(z)$, We get
\begin{align*}
    V_2(a,q)&= \frac{2 \pi N  i^k}{N^{\frac{1}{4}} \sqrt{q}} \sum_{\pm}  \sum_{n=1}^\infty \frac{\lambda_{f}(n)}{n^\frac{1}{4}} e\left(-n \frac{\overline{a}}{q}\right) \int_{\mathbb{R}} e\left(vy_2\right) \left(1+\frac{Hy_1}{Ny_2}\right)^{iT} W_{+}(y_1, y_2) e\left(\pm \frac{2 \sqrt{nNy_2}}{q}\right) \\& \times W_1^{\pm}\left(\frac{2 \sqrt{nNy_2}}{q}\right) y^{-1/4}_2 \,dy_2 \\
    &= \frac{2 \pi N  i^k}{N^{\frac{1}{4}} \sqrt{q}} \sum_{\pm}  \sum_{n=1}^\infty \frac{\lambda_{f}(n)}{n^\frac{1}{4}} e\left(-n \frac{\overline{a}}{q}\right) \int_{\mathbb{R}} e\left(vy_2\right)  e\left(\pm \frac{2 \sqrt{nNy_2}}{q}\right) \left(1+\frac{Hy_1}{Ny_2}\right)^{iT} U^{\pm}_{n,q}(y_1,y_2) \,dy_2,
\end{align*}
where we define the new smooth weight function
\begin{equation*}
   U^{\pm}_{n,q}(y_1,y_2)= W_{+}(y_1, y_2) W_1^{\pm}\left(\frac{2 \sqrt{nNy_2}}{q}\right) y^{-1/4}_2 .
\end{equation*}
Integration by parts(see Lemma~\ref{ibparts}) over the $y_2$ integral shows that the contribution is negligible unless $v \asymp \frac{\sqrt{nN}}{q}$, We conclude that the contribution of $n$-sum to $V_2(a,q)$ is negligible unless $n \asymp \frac{K^2Q^2}{N}$.
\subsection{Conclusion}
We conclude this section with the following estimate
\begin{align*}
    S(H,N) \ll \frac{HN^{3/2}}{|\mathfrak{Q}|K} \sum_{q \in \mathfrak{Q}} \frac{1}{q^2} \sum_{\substack{ h \sim \frac{TQ}{N} \\ (h,q)=1 }}  \sum_{ n \asymp \frac{K^2q^2}{N}} \frac{\lambda_f(n)}{n^{1/4}} \sum_{m \asymp \frac{K^2q^2}{N}}^\infty \frac{\lambda_f(m)}{m^{1/4}} e\left(\frac{\overline{h}(n-m) }{q}\right) \mathfrak{I}   +O(T^{-A}),
\end{align*}
where $\mathfrak{I}$ is an integral of the form(with possible sign changes)
\begin{align*}
\mathfrak{I}&=  \int_{\mathbb{R}}    
B\left(\frac{v}{K}\right) \times \int_{\mathbb{R}} e\left(-vy_3\right) e\left( \frac{2 \sqrt{mNy_3}}{q}\right) U_{m,q}(y_3) dy_3  \\ & \times
\int_{\mathbb{R}} \left(e\left(vy_2\right)e\left( - \frac{2 \sqrt{nNy_2}}{q}\right)\int_{\mathbb{R}} e\left(-\frac{hHy_1}{q}\right)\left(1+\frac{Hy_1}{Ny_2}\right)^{iT} e\left(\frac{vHy_1}{N}\right)U_{n,q}\left(y_1,y_2\right)  \right)  dy_1 dy_2
 dv.
\end{align*}
\section{Analysis of the Integral Transform}\label{eight}
The analysis of Integral amongst the possible sign changes will be similar. Throughout this section, By Example ~\ref{expo}, We absorb the non-oscillatory factors into the $T^\epsilon-$inert functions. 

To be specific, we work with the following integral. 
Define $$u=y_2-y_3+\frac{Hy_1}{N} \text{ }\text{ and } \text{ }v= \frac{2 \sqrt{mNy_3}}{q}- \frac{2 \sqrt{nNy_2}}{q}-\frac{hHy_1}{q}.$$
Then the integral can be written as
\begin{equation*}
\mathfrak{I}=\int_{\mathbb{R}} \int_{\mathbb{R}} \int_{\mathbb{R}} \int_{\mathbb{R}}  
    B\left(\frac{v}{K}\right) 
    e(uv+w) \left(1+\frac{Hy_1}{Ny_2}\right)^{iT}
    U_{m,q}(y_3) U_{n,q}\left(y_1,y_2\right)
    \,dv \,dy_1 \,dy_2 \,dy_3.
\end{equation*}
 Integration by parts over the $v$ integral shows that the contribution to $\mathfrak{I}$ is negligible unless $u \ll \frac{T^\epsilon}{K}$. Substitute $y_3=y_2+\frac{Hy_1}{N}-u$ to get
\begin{align*}
\mathfrak{I}&=\int_{\mathbb{R}} \int_{\mathbb{R}} \int_{\mathbb{R}} \int_{ u \ll \frac{T^\epsilon}{K}}  
    B\left(\frac{v}{K}\right)
    e(uv) 
    e\left(\frac{2 \sqrt{mNy_2(1+\frac{Hy_1}{Ny_2})}}{q}\right)  \\ & \times e\left(-\frac{2 \sqrt{nNy_2}}{q}-\frac{hHy_1}{q}\right) \left(1+\frac{Hy_1}{Ny_2}\right)^{iT}
    U_{m,q}(y_3) U_{n,q}\left(y_1,y_2\right)
      \,dy_1 \,dv \,dy_2 \,du.
\end{align*}
We focus on the $y_1$ integral, and we make a change of variables sending $y_1 \to y_1 y_2$, This would give
\begin{align*}
   \mathfrak{I}&=\int_{\mathbb{R}} \int_{\mathbb{R}} \int_{\mathbb{R}} y_2 \int_{ u \ll \frac{T^\epsilon}{K}}  
    U_{m,q}(y_3) U_{n,q}\left(y_1y_2,y_2\right) B\left(\frac{v}{K}\right) 
    e(uv) 
    e\left(\frac{2 \sqrt{mNy_2(1+\frac{Hy_1}{N})}}{q}\right)  \\ & \times e\left(- \frac{2 \sqrt{nNy_2}}{q}-\frac{hHy_1y_2}{q}\right) \left(1+\frac{Hy_1}{N}\right)^{iT}
     \,dy_1\,dv\,dy_2 \,du.
\end{align*}
We use the following Taylor series expansion for $\log(1+x)$
\begin{align*}
    \log(1+x)= x-\frac{x^2}{2}+\frac{x^3}{3}-\frac{x^4}{4} \dots,
\end{align*}
and we get
\begin{align*}
    \left(1+\frac{Hy_1}{N}\right)^{iT}=e\left(\frac{T}{2 \pi} \log\left(1+\frac{Hy_1}{N}\right)\right)=e\left(\frac{HTy_1}{2\pi N}-\frac{H^2Ty^2_1}{4 \pi N^2}\right) e\left(O\left(\frac{TH^3y^3_1}{N^3}\right)\right).
\end{align*}
The last term in the product is flat. Thus we get
\begin{align*}
\mathfrak{I}&=\int_{\mathbb{R}} \int_{\mathbb{R}} \int_{\mathbb{R}} y_2 \int_{ u \ll \frac{T^\epsilon}{K}}  
    U_{m,q}(y_3) U_{n,q}\left(y_1y_2,y_2\right) B\left(\frac{v}{K}\right) e(uv) 
    e\left(\frac{2 \sqrt{mNy_2(1+\frac{Hy_1}{N})}}{q}\right)  \\ & \times e\left(- \frac{2 \sqrt{nNy_2}}{q}\right) e\left(-\frac{hHy_1y_2}{q}+\frac{HTy_1}{2\pi N}-\frac{H^2Ty^2_1}{4 \pi N^2}\right) 
     \,dy_1\,dv\,dy_2 \,du \\
     &=\int_{\mathbb{R}} \int_{\mathbb{R}} \int_{\mathbb{R}} y_2 \int_{ u \ll \frac{T^\epsilon}{K}}  
    U_{m,q}(y_3) U_{n,q}\left(y_1y_2,y_2\right) B\left(\frac{v}{K}\right) e(uv) 
    e\left(  \frac{2\sqrt{mNy_2}}{q}-\frac{2 \sqrt{nNy_2}}{q}\right)\\ & \times e\left(\frac{Hy_1\sqrt{mNy_2}}{Nq}-\frac{hHy_1y_2}{q}\right)  e\left( O\left(\frac{2\sqrt{mNy_2}}{q} \frac{H^2y^2_1}{N^2}\right)\right)   e\left(\frac{HTy_1}{2\pi N}-\frac{H^2Ty^2_1}{4 \pi N^2}\right) 
     \,dy_1\,dv\,dy_2 \,du\\
     &=\int_{\mathbb{R}} \int_{\mathbb{R}} \int_{\mathbb{R}} y_2 \int_{ u \ll \frac{T^\epsilon}{K}}  
    U_{m,q}(y_3) U_{n,q}\left(y_1y_2,y_2\right) B\left(\frac{v}{K}\right) e(uv) 
    e\left(  \frac{2\sqrt{mNy_2}}{q}-\frac{2 \sqrt{nNy_2}}{q}\right)\\ & \times e\left(\frac{Hy_1\sqrt{mNy_2}}{Nq}-y_1 y_3 \frac{Hh}{q}\right)   e\left(-\frac{H^2Ty^2_1}{4 \pi N^2}\right) 
     \,dy_1\,dv\,dy_2 \,du .
\end{align*}
The term involving $e\left(O\left(\frac{2\sqrt{mNy_2}}{q} \frac{H^2y^2_1}{N^2}\right)\right)$ is flat and we absorb it into the weight function. An Integration by parts argument shows that the integral over $y_1$ is negligible unless 
\begin{align*}
    \frac{Hh}{q}\left(-y_2+\frac{qT}{2 \pi h N}\right) \ll \frac{T}{M^2} 
 \implies
 -y_2+\frac{qT}{2 \pi h N} \ll\frac{TQ}{M^2HH_0}.
 \end{align*}
 Substitute $y_2=y_3+\frac{qT}{2 \pi h N}$ with $y_3\ll \frac{TQ}{M^2HH_0}$. This would imply
\begin{align*}
    e\left(\frac{Hy_1}{q}\left(\frac{\sqrt{mNy_2}}{N}\right)\right)&=e\left(\frac{H\sqrt{mN}y_1}{qN}\sqrt{\frac{qT}{2\pi N h}}\right)e\left(\sqrt{\frac{qTm}{2 \pi h}}\frac{y_3 \pi h Hy_1 }{q^2T}\right),
\end{align*}
and
\begin{equation*}
        e\left(\frac{2 \sqrt{Ny_2}}{q}(\sqrt{m}-\sqrt{n})\right)=e\!\left(2 \sqrt{\frac{mT}{2 \pi qh }}\right)e\!\left(-2 \sqrt{\frac{nT}{2 \pi qh }}\right) e\!\left(N y_3\sqrt{\frac{2\pi hm}{q^3T}}\right)e\!\left(-N y_3\sqrt{\frac{2\pi hn}{q^3T}}\right).
\end{equation*}
Thus we get
\begin{align*}
 \mathfrak{I}&=\int_{\mathbb{R}} \int_{\mathbb{R}}  \int_{ u \ll \frac{T^\epsilon}{K}} \int_{ y_3 \ll \frac{T^\epsilon}{M}}  
    y_2 U_{m,q}(y_3) U_{n,q}\left(y_1y_2,y_2\right) B\left(\frac{v}{K}\right) e(uv)  \\ & \times
    e\!\left(2 \sqrt{\frac{mT}{2 \pi qh }}-2 \sqrt{\frac{nT}{2 \pi qh }}+N y_3\sqrt{\frac{2\pi hm}{q^3T}}-N y_3\sqrt{\frac{2\pi hn}{q^3T}}\right) \\ & \times e\left(\frac{H\sqrt{mN}y_1}{qN}\sqrt{\frac{qT}{2\pi N h}}+\sqrt{\frac{qTm}{2 \pi h}}\frac{y_3 \pi h Hy_1 }{q^2T}-\frac{hHy_1 y_3}{q}\right)     e\left(-\frac{H^2Ty^2_1}{4 \pi N^2}\right) 
     \,dy_3\,dy_1\,dv \,du.
\end{align*} 
In the region $ y_3 \ll \frac{N}{HT} $, the $y_3$ integral is flat, thus we assume that $ N/HT \ll y_3$. We make a change of variable sending $y_3 \to \frac{TQ}{M^2HH_0} y_3$, This gives
\begin{align*}
        \mathfrak{I}&=\frac{TQ}{M^2HH_0}\int_{\mathbb{R}} \int_{\mathbb{R}}  \int_{ u \ll \frac{T^\epsilon}{K}} \int_{ y_3 \ll T^\epsilon}  
    y_2 U_{m,q}(y_3) U_{n,q}\left(y_1y_2,y_2\right) B\left(\frac{v}{K}\right) e(uv)  \\ & \times
    e\!\left(2 \sqrt{\frac{mT}{2 \pi qh }}-2 \sqrt{\frac{nT}{2 \pi qh }} \right) e\left( -  \frac{hH}{q} \frac{TQ}{M^2Hh} y_3 \left(y_1-\frac{N}{H}\sqrt{\frac{2\pi m}{qhT}}+\frac{N}{H}\sqrt{\frac{2\pi n}{hqT}} \right)\right) \\ & \times e\left(\frac{H\sqrt{mN}y_1}{qN}\sqrt{\frac{qT}{2\pi N h}}\right)     e\left(-\frac{H^2Ty^2_1}{4 \pi N^2}\right) 
     \,dy_3 \,dy_1\,dv \,du.
\end{align*}
Integration by parts shows that the integral is negligible unless 
\begin{align*}
z=y_1-\frac{N}{H}\sqrt{\frac{2\pi m}{qhT}}+\frac{N}{H}\sqrt{\frac{2\pi n}{hqT}} \ll \frac{M^2}{T}.
\end{align*}
Thus we assume that $z \ll \frac{M^2}{T}$ and substitute $y_1=z+\frac{Nq}{H}\sqrt{\frac{2\pi m}{qhT}}-\frac{N}{H}\sqrt{\frac{2\pi n}{hqT}}$. Then one gets 
\begin{equation*}
    e\left(\frac{H\sqrt{mN}y_1}{qN}\sqrt{\frac{qT}{2\pi N h}}\right)e\left(-\frac{H^2Ty^2_1}{4 \pi N^2}\right)=e\left(\frac{m-n}{2qh}\right) e(O(1)).
\end{equation*}
Thus
\begin{align*}
        \mathfrak{I}&=\frac{TQ}{M^2HH_0}\int_{\mathbb{R}} \int_{\mathbb{R}}  \int_{ u \ll \frac{T^\epsilon}{K}} \int_{\frac{M^2}{T}}^{  T^\epsilon}  
     U^\pm_{m,q}(y_3) U^\pm_{n,q}\left(y_1y_2,y_2\right) B\left(\frac{v}{K}\right) e(uv) \\ & \times
    e\!\left(2 \sqrt{\frac{mT}{2 \pi qh }} \right)e\left(-2 \sqrt{\frac{nT}{2 \pi qh }} \right) e\left( -\frac{hNz y_3}{M^2q}\right)  e\left(\frac{m-n}{2qh}\right)
     \,dy_3\,dy_1\,dv \,du.
\end{align*}
We conclude this section with the following estimate on $S(H,N)$. Combining the analysis of the integral yields
\begin{align*}
    S(H,N) \ll & \frac{T^{1+\epsilon}N^\frac{3}{2}}{M^2KH_0}   \sum_{q \sim Q} \frac{1}{q^2} \sum_{h}  \sum_{m \asymp \frac{K^2Q^2}{N}} e\left( \frac{2\sqrt{mT}}{\sqrt{2 \pi q h}} \right)e\left( \frac{m}{2qh}\right)e\left(\frac{-\overline{h}m}{q} \right) \frac{\lambda_f(m)}{m^{1/4}} \\ & \times \sum_{n \asymp \frac{K^2Q^2}{N}} e\left( \frac{-2\sqrt{nT}}{\sqrt{2 \pi q h}} \right)e\left( \frac{-n}{2qh}\right)e\left(\frac{\overline{h}n}{q} \right) \frac{\lambda_f(n)}{n^{1/4}} \\ & \times  \int_{u} \int_{v} B\left(\frac{v}{K}\right) e(uv) \int_{y_3} \int_z e\left(-\frac{hN}{M^2q}z y_3\right) U'(m,n,q,h,y_3,z) \,dy_3 \,dz + O(T^{-A}),
\end{align*}
where $U'(m,n,q,h,y_3,z)$ is an $T^\epsilon$-inert function. Before proceeding with an application of Cauchy's inequality, we need to separate the variables $m$ and $n$ in the weight function $U'$. We use the Fourier inversion trick to achieve that since the variables $m$ and $n$ are related by a linear relation.
\section{Fourier Inversion- Separation of Variables}\label{nine}
Our weight function has the form $
U_{m,q}(y_3)\,U_{n,q}(y_1y_2,y_2)
$. Both factors depend on the variables $m$ and $n$ through the relations
\[
y_3=y_2+\frac{Hy_1}{N}-u,
\qquad
y_1=z+\frac{N}{H}\sqrt{\frac{2\pi m}{qhT}}
-\frac{N}{H}\sqrt{\frac{2\pi n}{hqT}}.
\]
Since $m$ and $n$ are related linearly, We apply Fourier Inversion to separate the dependency on both $m$ and $n$. We write
\begin{align*}
    U_{n,q}\left(y_1y_2,y_2\right)& = \int_{|\eta | \ll T^\epsilon} \hat{U}_{n,q}\left( \eta y_2,y_2\right) e(-\eta y_1) d\eta + O(T^{-A}) \\
    & = \int_{|\eta | \ll T^\epsilon} \hat{U}_{n,q}\left( \eta y_2,y_2\right) e(-\eta z)e\left(-\eta \frac{N}{H}\sqrt{\frac{2\pi m}{qhT}}\right)e\left( \eta \frac{N}{H}\sqrt{\frac{2\pi n}{hqT}}\right) d\eta + O(T^{-A}),
\end{align*}
Since $U_{n,q}$ is a $T^\varepsilon$-inert function, its Fourier transform decays rapidly. Hence the contribution from $|\eta|\gg T^\varepsilon$ is $O(T^{-A})$ for every $A>0$. An analogous decomposition also holds for $U_{m,q}(y_3)$. Thus we obtain the following upper bound for $S(H,N)$:
\begin{align*}
    S(H,N) \ll \frac{TN^\frac{3}{2}}{M^2KH_0} 
    \int_u \int_v B\left(\frac{v}{K}\right) e(uv) \int_{\substack{z \\ y_3 \sim T^\epsilon }}  \int_{\substack{\eta_1 \ll T^\epsilon  \\ \eta_2 \ll T^\epsilon}} \sum_q \sum_h \frac{e\left(\frac{Hhz y_3}{Mq}\right)}{q^2} \left(\sum_m \dots \right) \left(\sum_n \dots \right).
\end{align*}
Application of Cauchy's inequality to the $q$ and $h$ sum yields
\begin{align*}
    S(H,N) \ll \frac{TN^\frac{3}{2}}{M^2KH_0} 
    \int_{u \ll \frac{1}{K}} \int_{v \ll K }B\left(\frac{v}{K}\right) e(uv) \int_{y_3 \sim 1} \int_{z \ll \frac{M^2}{T}}  \int_{\substack{\eta_1 \ll T^\epsilon  \\ \eta_2 \ll T^\epsilon}} S_1^\frac{1}{2} S_2^\frac{1}{2} \,d\eta_1 \,d\eta_2 \,dz \,dy_3 \,dv \,du,
\end{align*}
where
\begin{align*}
   S_1=\sum_q \sum_h \Big|\sum_{n \asymp \frac{K^2Q^2}{N}} e\left(  \frac{-2 \sqrt{nT}}{\sqrt{2 \pi q h}}  \right) e\left(   \frac{-n}{2q h}   \right) e\left( \frac{\overline{h}n}{q}\right) \frac{\lambda_f(n)}{n^\frac{1}{4}} U''_{n,q}(\eta_1) e\left(\frac{N}{H}\sqrt{\frac{2\pi n }{hqT}} \eta \right) \Big|^2,
\end{align*}
and
\begin{align*}
   S_2=\sum_q \sum_h \Big|\sum_{m \asymp \frac{K^2Q^2}{N}} e\left( \frac{2\sqrt{mT}}{\sqrt{2 \pi q h}} \right)e\left( \frac{m}{2qh}\right)e\left(\frac{-\overline{h}m}{q} \right) \frac{\lambda_f(m)}{m^{1/4}} U''_{m,q}(\eta_2) e\left(-\frac{N}{H}\sqrt{\frac{2\pi m }{hqT}} \eta\right) \Big|^2.
\end{align*}
where $\eta=\eta_1+\eta_2$. The analysis for both $S_1$ and $S_2$ will be simmilar, so we focus on the analysis of $S_1$.
\section{Analysis of $S_1$}\label{ten}
In this section, we focus on the analysis of
\begin{align*}
   S_1=\sum_q \sum_h \Big|\sum_{n \ll \frac{K^2Q^2}{N}} e\left(  \frac{-2 \sqrt{nT}}{\sqrt{2 \pi q h}}  \right) e\left(   \frac{-n}{2q h}   \right) e\left( \frac{\overline{h}n}{q}\right) \frac{\lambda_f(n)}{n^\frac{1}{4}}  U''_{n,q}(\eta_1) e\left(-\frac{N}{H}\sqrt{\frac{2\pi n }{hqT}} \eta\right) \Big|^2.
\end{align*}
Note that by Ramanujan Bounds on Average.
\begin{align*}
 \sum_{n \asymp \frac{K^2Q^2}{N}}    \Big| \frac{\lambda_f(n)}{n^\frac{1}{4}} \Big|^2 \ll \left(\frac{K^2Q^2}{N}\right)^{-\frac{1}{2}}  \sum_{n \ll \frac{K^2Q^2}{N}}    |\lambda_f(n)|^2 \ll \left(\frac{K^2Q^2}{N}\right)^{\frac{1}{2}+\epsilon}.
\end{align*}
Now an application of Duality(See Lemma \ref{dualitylemma}), gives
\begin{align*}
   S_1 & \ll  \frac{KQ}{\sqrt{N}}\sup_{\|\alpha\|=1} \sum_{n \asymp \frac{K^2Q^2}{N}} \Bigg| \sum_q \frac{1}{q} \sum_h \alpha(q,h) e\left(  \frac{-2 \sqrt{nT}}{\sqrt{2 \pi q h}}  \right) e\left(   \frac{-n}{2q h}   \right) e\left( \frac{\overline{h}n}{q}\right) e\left(-\frac{N}{H}\sqrt{\frac{2\pi n }{hqT}} \eta \right) X(q,h,n) \Bigg|^2,
\end{align*}
where $X(q,h,n)$ are $T^\epsilon$-inert functions. We expand the absolute value squared, this would yield
\begin{align}\label{s1poisson}
    S_1 &\ll \frac{KQ}{\sqrt{N}}\sup_{\|\alpha\|=1} \sum_{q,q'} \frac{1}{qq'} \sum_{h,h'} \alpha(q,h) \overline{\alpha}(q',h') \sum_{n \asymp \frac{K^2Q^2}{N}}  F(n) e\left( n\frac{\overline{h}q'-\overline{h'}q}{qq'}\right),
\end{align}
where 
\begin{align*}
    F(n)= e\left(  \frac{-2 \sqrt{nT}}{\sqrt{2 \pi q h}} + \frac{2 \sqrt{nT}}{\sqrt{2 \pi q' h'}}  \right)  e\left(   \frac{-n}{2q h} + \frac{n}{2q' h'}   \right) & e\left(\frac{ \sqrt{2}N \pi \eta}{HT} \left(\sqrt{\frac{ T n }{q'h' \pi }}  - \sqrt{\frac{ T n }{qh \pi }}\right) \right) \\ & \times X(q,h,n) X(q',h',n),
\end{align*} 
Define $\Delta=q'h'-qh$. We separate the terms where $\Delta=0$ and $\Delta \neq 0$. When $\Delta=0$, We trivially estimate the summation which gives the contribution to $S_1$ to be
\begin{align*}
      S_1(\Delta=0) &\ll \frac{KQ}{\sqrt{N}}\sup_{\|\alpha\|=1} \sum_{q,q'} \frac{1}{qq'} \sum_{h,h'} |\alpha(q,h)|^2 \sum_{n \asymp \frac{K^2Q^2}{N}}  F(n) e\left( n\frac{\overline{h}q'-\overline{h'}q}{qq'}\right) \\
      &\ll \frac{KQ}{\sqrt{N}} \frac{K^2Q^2}{N} \frac{1}{Q^2}\sup_{\|\alpha\|=1} \sum_{q} \sum_{h} |\alpha(q,h)|^2 \sum_{\substack{q' , h'  \\ q'h'=qh}} 1 \\
      &\ll \frac{KQ}{\sqrt{N}} \frac{K^2Q^2}{N} \frac{1}{Q^2} T^\epsilon = \frac{K^3Q}{N^\frac{3}{2}} T^\epsilon.   
\end{align*}
Thus we conclude that
\begin{equation}\label{s1bounds}
      S_1(\Delta=0) \ll  \frac{K^3Q}{N^\frac{3}{2}} T^\epsilon.
\end{equation}
When $\Delta \neq 0,$ We apply Poisson summation formula(see Lemma $(\ref{poissonmodq})$ to the $n$-sum in equation $\ref{s1poisson}$. Thus
\begin{align*}
    \sum_{n \asymp \frac{K^2Q^2}{N}} F(n) e\left( n\frac{\overline{h}q'-\overline{h'}q}{qq'}\right) &= \sum_{n \in \mathbb{Z}} U\left(\frac{nN}{K^2Q^2}\right)F(n) e\left( n\frac{\overline{h}q'-\overline{h'}q}{qq'}\right) \\
    & =\sum_{ \substack{ n \in \mathbb{Z} \\ n \equiv \overline{h'}q-\overline{h}q' \,\text{mod}\, qq'} } J_n,
\end{align*}
where \begin{equation*}
    J_n=\frac{K^2Q^2}{N} \int_{\mathbb{R}} U(t) e\left(\A t -\B \sqrt{t}\right) \,dt,
\end{equation*}
where we define
\begin{equation*}
    \A=-\frac{K^2Q^2}{N}\left(\frac{n}{qq'}+\frac{1}{qh}-\frac{1}{q'h'}\right),
\end{equation*}
and
\begin{equation*}
    \B=\frac{2KQ\sqrt{T}}{\sqrt{2 \pi N}}  \left(1 +\frac{N \pi \eta}{HT} \right)\left(\frac{1}{\sqrt{qh}}-\frac{1}{\sqrt{q'h'}} \right).
\end{equation*}
We conclude this section with the following result.
\begin{lemma}\label{bounds1}
We have
\begin{equation*}
    S_1 = S_1(\Delta=0) + S_1(\Delta \neq 0),
\end{equation*}
where
\begin{equation}\label{s1bounds}
    S_1(\Delta =0 ) \ll  \frac{K^3Q}{N^\frac{3}{2}} T^\epsilon,
\end{equation}
and
\begin{align*}
    S_1(\Delta \neq 0) & \ll  T^\epsilon\frac{K^3Q^3}{N^\frac{3}{2}}\sup_{\|\alpha\|_2=1} \sum_{q,q'} \frac{1}{qq'} \sum_{\substack{h,h' \\ \Delta \neq 0}} \alpha(q,h) \overline{\alpha}(q',h') \sum_{ \substack{ n \in \mathbb{Z} \\ n \equiv \overline{h'}q-\overline{h}q' \,\text{mod}\, qq'} } \int_{\mathbb{R}} U(t) e\left(  \A t - \B  \sqrt{t} \right) \,dt. 
\end{align*}
		where the supremum is taken over all sequences of complex numbers $\alpha(q,h)$ such that $$\|\alpha\|_2=\sqrt{\sum_{q,h}|\alpha(q,h)|^2}=1.$$
\end{lemma}

\section{Phase Analysis}\label{eleven}

To get bounds on $S_1(\Delta \neq 0)$, we analyze the oscillatory integral
\begin{equation*}
    \int_{\mathbb{R}} U(t)e\left( \A t-\B \sqrt{t}\right)\,dt.
\end{equation*}

Recall that
\begin{equation*}
    \A
    =
    -\frac{K^2Q^2}{N}
    \left(
    \frac{n}{qq'}
    +
    \frac{1}{qh}
    -
    \frac{1}{q'h'}
    \right)
    =
    -\frac{K^2Q^2}{N}
    \left(
    \frac{n}{qq'}
    +
    \frac{\Delta}{(QH_0)^2}
    \right),
\end{equation*}
and
\begin{equation*}
    \B
    =
    \frac{2KQ\sqrt{T}}{\sqrt{2\pi N}}
    \left(
    \frac{1}{\sqrt{qh}}
    -
    \frac{1}{\sqrt{q'h'}}
    \right)
    \asymp
    \frac{KQ\sqrt{T}}{\sqrt{N}} \left(1+\frac{N \pi \eta }{HT}\right)
    \frac{\Delta}{(QH_0)^{3/2}}.
\end{equation*}
Since $|\eta|\ll T^\epsilon,\ N\ll T^{1+\epsilon},\ \text{and}\ H\gg T^\epsilon$, we have $N \pi \eta/HT \ll o(1).$
Hence
\[
\B
\asymp
\frac{KQ \sqrt{T}}{\sqrt{N}}
\left(
\frac{1}{\sqrt{qh}}
-
\frac{1}{\sqrt{q'h'}}
\right).
\]
Suppose first that
\[
n \gg \frac{\Delta}{H_0^{2}}.
\]
In this range, the term $n/qq'$ dominates the contribution to $\A$. By Lemma ~\ref{ibparts}, the integral is negligible unless
\[
|\A| \asymp |\B| .
\]
Consequently,
\begin{equation}\label{sizeofn}
    n
    \asymp
    \frac{Q\sqrt{TN}\,\Delta}
         {K(QH_0)^{3/2}}.
\end{equation}
Next suppose that
\[
n \ll \frac{\Delta}{H_0^{2}}.
\]
Then the second term in $\A$ dominates, and
\begin{equation*}
    \frac{|\B |}{|\A |}
    \asymp
    \frac{KQ\sqrt{T}\Delta}
         {\sqrt{N}(QH_0)^{3/2}}
    \cdot
    \frac{N(QH_0)^2}
         {K^2Q^2\Delta}
    =
    \frac{\sqrt{TNQH_0}}{KQ}
    =M.
\end{equation*}
Hence
\[
T^\epsilon |\A| \ll |\B|.
\]
Repeated integration by parts(See Lemma~\ref{ibparts}) implies that the integral is negligible whenever $|\B |\gg T^\epsilon. $ On the other hand, if $|\B |\ll T^\epsilon,$
then $\Delta\ll\frac{Q^2MT^\epsilon}{N}.$ Thus contribution to $S_1(\Delta\neq 0)$ from the range
\[
n\ll\frac{\Delta}{H_0^{2}}
\qquad \text{ and } \qquad
\Delta\ll\frac{Q^2MT^\epsilon}{N},
\]
is $O(T^{-A})$ since it forces $n = 0.$ We are therefore left with the case in which $\frac{n}{qq'}$ is the dominant term in $\A$. In this situation, We estimate the oscillatory integral appearing in Lemma~\ref{bounds1} using the second derivative test (see Lemma~\ref{secondd}). The phase function is
\[
\phi(t)=\A t-\B\sqrt{t},
\]
so that
\[
\phi''(t)
=
\frac{\B}{4t^{3/2}}
\asymp
\B.
\]
Hence, by Lemma~\ref{secondd},
\[
\left|
\int_{\mathbb{R}}
U(t)e\left(\A t-\B\sqrt{t}\right)\,dt
\right|
\ll
\B^{-1/2}.
\]
We partition the sum $S_1(\Delta \neq 0)$ into two disjoint cases according to the arithmetic relations among the variables $n$, $q$, $q'$, $h$, and $h'$, since each case requires a different method of estimation. When $n=0$, the congruence condition forces $q=q'$. Since $\Delta \neq 0$ implies $h\neq h'$. We denote its contribution to $S_1(\Delta \neq 0)$ by $S_{11}.$ 
Moving to the non-zero frequency case, When $n\neq0$, We consider the case when $q=q'$. Since $\Delta>0$, we necessarily have $h\neq h'$. The congruence condition implies that $q\mid n$. As $n \neq 0$, this gives $Q \ll |n|$. However since $\Delta \ll QH_0,$ thus equation \ref{sizeofn} implies, 
\begin{equation*}
    n \ll \frac{NMT^\epsilon}{T} \ll \frac{N T^\epsilon}{K} \ll Q.
\end{equation*}
Our choice $N/K \ll Q$ forces $n=0$, This proves that the contribution from $n \neq 0$ and $q= q'$ is negligible. On the other hand, We denote the contribution of $n \neq 0$ and $q \neq q'$ to $S_1(\Delta \neq 0)$ by $S_{12}$. Thus we have $$S_1(\Delta \neq 0)=S_{11}+S_{12}.$$
\subsection{Zero Frequency}- When $n=0$, then the condition $0 \equiv \overline{h'}q-\overline{h}q' \,\text{mod}\, qq'$ implies $ q$ divides $ q'$, Since $q$ and $q'$ are primes, this forces $q=q'$ and thus $h \equiv h' \,\text{mod}\, q$. Since $\Delta \neq 0$ implies $h \neq h'$, One can see that for $n=0$ and $q=q'$ implies 
\begin{equation*}
|\B| \gg \frac{KQ \sqrt{T}}{\sqrt{N}} \frac{Q^2 N^\frac{3}{2}}{T^\frac{3}{2} Q^3}=\frac{KN}{T}=\frac{N}{M} \gg T^{\frac{1}{6}-\epsilon},
\end{equation*}
and \begin{equation*}
    \frac{|\B|}{|\A|} \asymp \frac{KQ\sqrt{T}\Delta}{\sqrt{N}(QH_0)^{\frac{3}{2}}} \frac{N(QH_0)^2}{K^2Q^2 \Delta} = \frac{\sqrt{TNQH_0}}{KQ}=M
\end{equation*}
Since $\frac{|\B|}{|\A|} \gg M \gg T^\frac{1}{3}$, By an application of integration by parts, we conclude that
\begin{equation}\label{s11bounds}
    S_{11}\ll T^{-A} \text{ }\text{ for any }\text{ } A \geq 0.
\end{equation}
\subsection{Non-Zero Frequency}
Now we focus on the $S_{12}.$ Since $q$ and $q'$ are distinct primes, The congruence condition $$n \equiv \overline{h'}q-\overline{h}q' \,\text{mod}\, qq'$$ implies $(n,q)=(n,q')=1$  and thus $h \equiv -q'\overline{n} \,\text{mod}\, q $ and $h' \equiv q\overline{n} \,\text{mod}\, q'$
\begin{align*}
    S_{12} &\ll \frac{K^3Q^3}{N^\frac{3}{2}}\sup_{\|\alpha\|_2=1} \sum_{q} \sum_{q' \neq q} \frac{1}{qq'} \sum_{h,h'} |\alpha(q,h) \overline{\alpha}(q',h')| \sum_{ \substack{ n \in \mathbb{Z} \\ n \equiv \overline{h'}q-\overline{h}q' \,\text{mod}\, qq'} } B^{-\frac{1}{2}} 
\end{align*}
 We use Arithmetic-Geometric Mean to write $|\alpha(q,h) \alpha(q',h')|  \ll  |\alpha(q,h)|^2+|\alpha(q',h')|^2. $ By symmetry, it suffices to analyse
\begin{align*}
    S_{12} & \ll \frac{K^\frac{5}{2}Q^\frac{1}{2}}{N^\frac{5}{4} T^\frac{1}{4}}\sup_{\|\alpha\|_2=1} \sum_{q,q'}  \sum_{h,h'}  |\alpha(q,h)|^2 \sum_{ \substack{ n \ll \frac{Q\sqrt{TN} \Delta}{K(QH_0)^\frac{3}{2} } \\ n \equiv \overline{h'}q-\overline{h}q' \,\text{mod}\, qq'} } \left(\frac{1}{\sqrt{qh}}-\frac{1}{\sqrt{q'h'}} \right)^{-\frac{1}{2}}.
\end{align*}
Both $qh$ and $ q'h'$ has size $QH_0$ and rationalising the denominator gives \begin{equation*}
    \left(\frac{1}{\sqrt{qh}}-\frac{1}{\sqrt{q'h'}} \right)^{-\frac{1}{2}} \asymp  (QH_0)^{\frac{3}{4}} \Delta^{-\frac{1}{2}}.
\end{equation*}
Recall that $\Delta > 0$, We split the sum into dyadic intervals depending upon where $q'h'-qh$ lies. Thus
\begin{align*}
    S_{12} & \ll  T^\epsilon \frac{K^\frac{5}{2}Q^\frac{1}{2}(QH_0)^{\frac{3}{4}}}{N^\frac{5}{4} T^\frac{1}{4}}
    \sup_{\Delta \ll (QH_0)}
    \sup_{\|\alpha\|_2=1} 
\sum_{ \substack{ n \ll \frac{Q\sqrt{TN} \Delta}{K(QH_0)^\frac{3}{2} }} }
    \sum_{q} \sum_{q'}  \sum_{\substack{ h \\ h \equiv -q'\overline{n} \,\text{mod}\, q}} \sum_{\substack{ h' \\ q'h'-qh \sim \Delta \\ h' \equiv q\overline{n} \,\text{mod}\, q'}}  |\alpha(q,h)|^2  \Delta^{-\frac{1}{2}}.
\end{align*}
The term in the summation is independent of $q'$ and $h'$, so we take $|\alpha(q,h)|^2$ outside, This yields
\begin{align*}
    S_{12} & \ll T^\epsilon \frac{K^\frac{5}{2}Q^\frac{1}{2}(QH_0)^{\frac{3}{4}}}{N^\frac{5}{4} T^\frac{1}{4}}
    \sup_{\Delta \ll QH_0}
    \sup_{\|\alpha\|_2=1} 
\sum_{ \substack{ n \ll \frac{Q\sqrt{TN} \Delta}{K(QH_0)^\frac{3}{2} }} }
    \sum_{q} \sum_{q'}  \sum_{\substack{ h \\ h \equiv -q'\overline{n} \,\text{mod}\, q}}  |\alpha(q,h)|^2 \sum_{\substack{ h' \\ q'h'-qh \sim \Delta \\ h' \equiv q\overline{n} \,\text{mod}\, q'} } \Delta^{-\frac{1}{2}}.
\end{align*}
Since $q'h'-qh \sim \Delta$, $h'$ lies in an interval of length $\frac{\Delta}{q'}$. Thus we get
\begin{align}
    \sum_{\substack{h' \\ h' \equiv q\overline{n} \,\text{mod}\, q' \\ q'h'-qh \sim \Delta}} \Delta^{-\frac{1}{2}} \ll \left(1+\frac{\Delta}{q'}\right) \frac{1}{\sqrt{\Delta}}. 
    \end{align}
Thus we have
\begin{align*}
    S_{12} & \ll T^\epsilon \frac{K^\frac{5}{2}Q^\frac{1}{2}(QH_0)^{\frac{3}{4}}}{N^\frac{5}{4} T^\frac{1}{4}}
    \sup_{\Delta \ll QH_0} \left(1+\frac{\Delta}{q'}\right) \frac{1}{\sqrt{\Delta}}  
\sum_{ \substack{ n \ll \frac{Q\sqrt{TN} \Delta}{K(QH_0)^\frac{3}{2} }} }
   \sup_{\|\alpha\|_2=1} \sum_{q} \sum_{q'}  \sum_{\substack{ h \\ h \equiv -q'\overline{n} \,\text{mod}\, q}}  |\alpha(q,h)|^2.
\end{align*}
Now we focus on the following expression
\begin{align*}
       \sup_{\|\alpha\|_2=1} \sum_{q} \sum_{q'}  \sum_{\substack{ h \\ h \equiv -q'\overline{n} \,\text{mod}\, q}}  |\alpha(q,h)|^2 & =    \sup_{\|\alpha\|_2=1} \sum_{q} \sum_{h}   \sum_{\substack{ q' \\ q' \equiv -hn \,\text{mod}\, q}}  |\alpha(q,h)|^2 \\
       & =    \sup_{\|\alpha\|_2=1} \sum_{q} \sum_{h}   |\alpha(q,h)|^2  \sum_{\substack{ q' \\ q' \equiv -hn \,\text{mod}\, q}} 1 \\
       & =    \sup_{\|\alpha\|_2=1} \sum_{q} \sum_{h}   |\alpha(q,h)|^2 \\
       & =1 .
\end{align*}
Thus we conclude that
\begin{align*}
    S_{12} & \ll T^\epsilon \frac{K^\frac{5}{2}Q^\frac{1}{2}(QH_0)^{\frac{3}{4}}}{N^\frac{5}{4} T^\frac{1}{4}}
    \sup_{\Delta \ll QH_0} \left(1+\frac{\Delta}{q'q'}\right) \frac{1}{\sqrt{\Delta}}  
\sum_{ \substack{ n \ll \frac{Q\sqrt{TN} \Delta}{K(QH_0)^\frac{3}{2} }} } 1 \\
& \ll T^\epsilon \frac{K^\frac{5}{2}Q^\frac{1}{2}(QH_0)^{\frac{3}{4}}}{N^\frac{5}{4} T^\frac{1}{4}}
    \sup_{\Delta \ll QH_0} \left(1+\frac{\Delta}{q'q'}\right) \frac{1}{\sqrt{\Delta}}  
\frac{Q\sqrt{TN} \Delta}{K(QH_0)^\frac{3}{2} } \\
& \ll T^\epsilon \frac{K^\frac{5}{2}Q^\frac{1}{2}(QH_0)^{\frac{3}{4}}}{N^\frac{5}{4} T^\frac{1}{4}}
 \frac{QH_0}{Q^2} \sqrt{QH_0}  
\frac{Q\sqrt{TN} }{K(QH_0)^\frac{3}{2} } \\
& \ll T^\epsilon \frac{K^\frac{5}{2}Q^\frac{1}{2}(QH_0)^{\frac{3}{4}}}{N^\frac{5}{4} T^\frac{1}{4}}
 \frac{1}{Q}   
\frac{\sqrt{TN} }{K } \\
& \ll T^\epsilon \frac{K^\frac{3}{2}(QH_0)^{\frac{3}{4}}}{N^\frac{5}{4} T^\frac{1}{4}}
 \frac{\sqrt{TN} }{\sqrt{Q}} = T^\epsilon \frac{K^\frac{3}{2}TQ}{N^\frac{3}{2}}.
\end{align*}
Thus
\begin{equation}\label{s12bounds}
    S_{12} \ll T^\epsilon \frac{K^\frac{3}{2}TQ}{N^\frac{3}{2}}.
\end{equation}
\section{Conclusion}\label{twelve}
Combining the bounds of $S_1(\Delta=0)$ from $\ref{s1bounds}$ , bounds for $S_{11}$ from equation $\ref{s11bounds}$ and bounds for $S_{12}$ from equation $\ref{s12bounds}$ , The shifted convolution sum is bounded by
\begin{equation*}
  S(H,N) \ll  T^\epsilon \left( \frac{K^3Q}{N^\frac{3}{2}} + \frac{K^\frac{3}{2}TQ}{N^\frac{3}{2}} + T^{-A}  \right)\times \frac{M^2}{T} \times \frac{TN^\frac{3}{2}N}{M^2KTQ}=NT^\epsilon \left(
\frac{T}{M^2}+\frac{\sqrt{T}}{\sqrt{M}}\right)
\end{equation*}
 This proves Proposition $\ref{boundshn}$. 
 Substituting Proposition~\ref{boundshn} into equation ~\eqref{combi}, we obtain
 \begin{align}
S_f(M,T) &\ll MT^\epsilon +\sup_{N \ll T^{1+\epsilon}} \frac{M}{N}\sup_{ H \ll \frac{NT^{\epsilon}}{M} }   NT^\epsilon \left(
\frac{T}{M^2}+\frac{\sqrt{T}}{\sqrt{M}}\right) \\
& \ll T^\epsilon\left(M+\frac{T}{M}+\sqrt{TM}\right). 
 \end{align}

This proves Theorem ~\ref{theorem1}. To deduce the corresponding pointwise bound, we use an analogue of Lemma~3.5 of Aggarwal, Leung and Munshi~\cite{ALM}. The proof is identical, with minor modifications.
\begin{lemma}
Let $f$ be a holomorphic cusp form of fixed weight and level $1$.
For $X/2\le T \le X$, we have
\[
|L(\tfrac12+iT,f)|^2
\ll_f
\log X
\left(
1+
\int_{-\log X}^{\log X}
|L(\tfrac12+i(T+v),f)|^2
e^{-v^2/2}\,dv
\right).
\]
\end{lemma}
Applying the lemma with $X=T$, and using the fact that $M\gg T^{1/3}\gg\log T$, we obtain
\[
\int_{-\log T}^{\log T}
|L(\tfrac12+i(T+v),f)|^2\,dv
\le
\int_{-M}^{M}
|L(\tfrac12+i(T+v),f)|^2\,dv.
\]
This would give
\begin{align*}
|L\left(\tfrac12+iT,f\right)|^2 &  \ll_f
\log T 
\left(
1+
\int_{-\log T}^{\log T}
|L(\tfrac12+i(T+v),f)|^2
\,dv
\right)
\\ & \ll  \log T \left(1+\int_{-M }^{M} |L(\tfrac12+i(T+v),f)|^2\,dv\right) \\
& \ll T^{2/3+\epsilon}.
\end{align*}
Taking square roots gives
\[
L\!\left(\tfrac12+iT,f\right)
\ll_{f,\epsilon}
T^{1/3+\epsilon},
\]
which completes the proof of Corollary~\ref{theorem2}.
\section*{Acknowledgements}
The author is deeply grateful to his supervisor, Prof. Kummari Mallesham, for his constant encouragement, guidance, and many valuable discussions throughout this work. The author sincerely thanks Prof. Sumit Kumar for suggesting this problem, for his generous technical guidance, and for carefully reading several drafts of the manuscript. This work was supported by the GATE Fellowship.\bibliographystyle{plain}  
\bibliography{bbb}  
\nocite{*}
\end{document}